\documentclass[11pt]{article}
\usepackage{amsfonts}
\usepackage{amsmath}
\usepackage{mathrsfs}
\usepackage{color}
\usepackage{tikz}
\usepackage{tkz-graph}
\usepackage{diagbox}
\usepackage{hyperref}
\hypersetup{hypertex=true,
	colorlinks=true,
	linkcolor=blue,
	anchorcolor=blue,
	citecolor=blue}
\usetikzlibrary{chains,fit,shapes}
\usetikzlibrary{calc}
\usetikzlibrary{chains,fit,shapes}
\usetikzlibrary{positioning,calc}
\usepackage{array,booktabs}
\usetikzlibrary{arrows}

\usepackage{authblk}
\usepackage{mathrsfs,amscd,amssymb,amsthm,amsmath,bm,graphicx,psfrag,subfigure,url,indentfirst,verbatim}
\usepackage[skipabove=5pt,skipbelow=5pt,leftmargin=3pt,rightmargin=3pt,linewidth=0.8pt]{mdframed}
\newtheorem{thm}{Theorem}
\newtheorem{rem}{Remark}
\newtheorem{prop}{Proposition}

\newtheorem{lem}{Lemma}
\newtheorem{cor}{Corollary}

\begin{document}
	\title{\bf New tools to study 1-11-representation of graphs}
	\author[a]{Mikhail Futorny\thanks{Email: mfutorny@usp.br}}
	\author[b]{Sergey Kitaev\thanks{Email: sergey.kitaev@strath.ac.uk}}
	\author[c]{Artem Pyatkin\thanks{Email: artem@math.nsc.ru}}
	\affil[a]{Institute of Mathematics and Statistics, University of S\~{a}o Paulo, R.do Mat\~{a}o, 1010, Brazil}
	\affil[b]{Department of Mathematics and Statistics, University of Strathclyde, 26 Richmond Street, Glasgow G1 1XH, United Kingdom}
	\affil[c]{Sobolev Institute of Mathematics, Koptyug ave, 4, Novosibirsk, 630090, Russia}
	\date{}
	\maketitle

\begin{abstract}
The notion of a $k$-11-representable graph was introduced by Jeff Remmel in 2017 and studied by Cheon et al.\ in 2019 as a natural extension of the extensively studied notion of word-representable graphs, which are precisely 0-11-representable graphs. A graph $G$ is $k$-11-representable if it can be represented by a word $w$ such that for any edge (resp., non-edge) $xy$ in $G$ the subsequence of $w$ formed by $x$ and $y$ contains at most $k$ (resp., at least $k+1$) pairs of consecutive equal letters. A remarkable result of Cheon at al.\ is that {\em any} graph is 2-11-representable, while it is unknown whether every graph is 1-11-representable. Cheon et al.\ showed that the class of 1-11-representable graphs is strictly larger than that of word-representable graphs, and they introduced a useful toolbox to study 1-11-representable graphs.

In this paper, we introduce new tools for studying 1-11-representation of graphs. 
We apply them for establishing 1-11-representation of Chv\'{a}tal graph,  Mycielski graph,  split graphs, and graphs whose vertices can be partitioned into a comparability graph and an independent set.  
\end{abstract}

\vspace{2mm} \noindent{\bf Keywords}: 1-11-representable graph, word-representable graph, Chv\'{a}tal graph, split graph, Mycielski graph, comparability graph
\vspace{2mm}

\section{Introduction}

\noindent
Various ways to represent graphs have evolved into a field of study, interesting from both mathematical and computer science perspectives \cite{Spinrad}. Of more relevance to us is the theory of word-representable graphs, \cite{KL15}, admitting a myriad of various generalizations. The basic idea here is to encode a given graph by a word using specified rules for defining edges/non-edges. For example, in the word-representable graphs alternations of letters in words define edges/non-edges, whilst this idea has been generalized by utilizing other patterns \cite{JKPR}. A given graph may, or may not admit representation under a given set of rules, so the main concern in the area of interest to us is whether a given graph is representable. Other research questions may include  studying algorithmic aspects of representations, its minimal lengths, connections to other structures like graph orientations, applications, etc. 

A particular way to represent graphs is $k$-11-representation introduced by Jeff Remmel in 2017 and studied by Cheon et al.\ in~\cite{CKKKP2019}. This way to represent graphs, formally defined in Section~\ref{k-11-repr-sub}, is a natural way to generalize the notion of a word-representable graph that are precisely 0-11-representable graphs. Remarkably, {\em any} graph is 2-11-representable and the class of 1-11-representable graphs is strictly larger than that of 0-11-representable graphs (i.e.\ word-representable graphs); see~\cite{CKKKP2019}. 
It is still unknown whether there exist graphs that are not 1-11-representable. Clearly, such graphs (if they exist) must be non-word-representable.
Hence, proving that various classes of non-word-representable graphs are 1-11-representable is a worthwhile direction of research.  

\subsection{Our results and organization of the paper}

\noindent
In this paper, we observe the need of introducing new tools to study 1-11-representable graphs as the known set of tools does not allow to establish 1-11-representation of some  known non-word-representable graphs. In particular, we introduce a new tool for establishing 1-11-representation of the Chv\'{a}tal graph and another tool for proving that every split graph is 1-11-representable. We also generalize these tools to prove 1-11-representability for certain more general classes of graphs. Finally, we revisit the proof in~\cite{CKKKP2019} that every graph on at most 7 vertices is 1-11-representable to fill in the gap in the proof caused by usage of an incomplete list of small non-word-representable graphs, where two graphs were missing.

The paper is organized as follows. In Section~\ref{prelim-sec} we introduce all (classes of) graphs considered  in this paper highlighting in separate subsections more important word-representable graphs and related to them semi-transitive orientations (Subsection~\ref{wrg-sec}) and $k$-11-representable graphs (Subsection~\ref{k-11-repr-sub}). Also, in Subsection~\ref{known-tools-sec} we provide a comprehensive list of known results about 1-11-representable graphs that provide a powerful base to study 1-11-representation of graphs. In Section~\ref{new-tool-sec} we introduce new tools to study 1-11-representable graphs and discuss its applications for the Chv\'{a}tal graph in Subsection~\ref{Chvatal-sub} and for split graphs and for graphs whose vertices can be partitioned into a comparability graph and an independent set  in Subsection~\ref{split-sub}. Also, we complete justification of the fact that all graphs on at most 7 vertices are 1-11-representable in Subsection~\ref{7-verti-sub} and provide concluding remarks in 
 Section~\ref{final-sec}. 

\section{Preliminaries}\label{prelim-sec}

\noindent
We begin with defining (classes of) graphs appearing in this paper under various contexts.
Throughout this paper, we denote by $G\setminus v$ the graph obtained from a graph $G$ by deleting a vertex $v\in V(G)$ and all edges adjacent to it. Also, for any
$A\subseteq V$ and $v\in V$ let $N_A(v):=\{ u\in A\ |\ uv\in E\}$, that is, $N_A(v)$ is the set of {\em neighbours} of $v$ in $A$. If $A=V$ we write simply $N(v)$. We use the notation $G[A]$ for the subgraph of $G$ induced by the subset $A$.

A {\em circle graph} is the intersection graph of a set of chords of a circle, i.e.\ it is an undirected graph whose vertices can be associated with chords of a circle such that two vertices are adjacent if and only if the corresponding chords cross each other \cite{S1994}. An {\em interval graph} has one vertex for each interval in a family of intervals on a line, and an edge between every pair of vertices corresponds to intervals that intersect \cite{LB1962}. 
A {\em split graph} is a graph in which the vertices can be partitioned into a clique and an independent set \cite{G80,HS1981}.  
For an arbitrary graph $G=(V,E)$ with $V=\{v_1, \ldots, v_n\}$, define the $Mycielski$ graph $\mu(G)=(V\cup U\cup \{x\}, E\cup E')$ where $U=\{u_1,\ldots, u_n\}$ and
$$E'= \cup_{i=1}^n(\{ xu_i\}\cup \{ yu_i \mbox{ for all } y\in N_V(v_i)\}).$$
In other words, $\mu(G)$ contains $G$ itself as a subgraph, the independent set consisting of a copy of each its vertex, and a vertex $x$ adjacent to all these copies. For example, the graphs $\mu(C_{3})$, $\mu(C_{4})$, and $\mu(C_{5})$ are in  Figure~\ref{Mycielski}. 
The importance of Mycielski graphs follows from the well-known fact \cite{Myc} that this construction allows to increase the chromatic number of a triangle-free graph without adding new triangles (i.e if $G$ is a triangle-free $k$-chromatic graph then $\mu(G)$ is a triangle-free $(k+1)$-chromatic graph).  The Chv\'{a}tal graph is presented to the left in  Figure~\ref{Chvatal-graph}. 

An orientation of a graph is {\em transitive}, if the presence of the edges $u\rightarrow v$ and $v\rightarrow z$ implies the presence of the edge $u\rightarrow z$. An undirected graph $G$ is a {\em comparability graph} if $G$ admits a transitive orientation. 

\subsection{Word-representable graphs and semi-transitive orientations}\label{wrg-sec}

\noindent 
Two letters $x$ and $y$ alternate in a word $w$ if after deleting in $w$ all letters but the copies of $x$ and $y$ we either obtain a word $xyxy\cdots$ or a word $yxyx\cdots$ (of even or odd length).  A graph $G=(V,E)$ is {\em word-representable} if and only if there exists a word $w$
over the alphabet $V$ such that letters $x$ and $y$, $x\neq y$, alternate in $w$ if and only if $xy\in E$. 
The unique minimum (by the number of vertices) non-word-representable graph on 6 vertices is the wheel graph $W_5$, while there are 25 non-word-representable graphs on 7 vertices. We note that the original list of 25 non-word-representable graphs on 7 vertices presented, for example, in~\cite{KL15} contains two incorrect graphs, so we refer to \cite{KitaevSun} for the corrected catalog of the 25 graphs. 

A graph is {\em permutationally representable} if it can be represented by concatenation of permutations of (all) vertices. Thus, the class of permutationally representable graphs is a subclass of word-representable graphs. The following theorem classifies these graphs.

\begin{thm}[\cite{KL15}]\label{comp-thm} A graph is permutationally representable if and only if it is a comparability graph. \end{thm} 

An orientation of a graph is {\em semi-transitive} if it is acyclic, and for any directed path $v_0\rightarrow v_1\rightarrow \cdots \rightarrow v_k$ either there is no arc from $v_0$ to $v_k$, or $v_i\rightarrow v_j$ is an arc for all $0\leq i<j\leq k$. An induced subgraph on at least four vertices $\{v_0,v_1,\ldots,v_k\}$ of an oriented graph is a {\em shortcut} if it is acyclic, non-transitive, and contains both the directed path $v_0\rightarrow v_1\rightarrow \cdots \rightarrow v_k$ and the arc $v_0\rightarrow v_k$, that is called the {\em shortcutting edge}. A semi-transitive orientation can then be alternatively defined as an acyclic shortcut-free orientation. A fundamental result in the area of word-representable graphs is the following theorem.

\begin{thm}[\cite{Halldorsson}]\label{semi-trans-thm} A graph is word-representable if and only if it admits a semi-transitive orientation. \end{thm} 

For instance, it follows from Theorem~\ref{semi-trans-thm} that each 3-colorable graph is word-representable (just direct each edge from a lesser color to a larger one). 

\subsection{$k$-11-representable graphs}\label{k-11-repr-sub}

\noindent
A {\em factor} in a word $w_1w_2\ldots w_n$ is a word $w_iw_{i+1}\ldots w_j$ for $1\leq i\leq j\leq n$. For any word $w$, let $\pi(w)$ be the {\em initial permutation} of $w$ obtained by reading $w$ from left to right and recording the leftmost occurrences of the letters in $w$. Denote by $r(w)$ the {\em reverse} of $w$, that is, $w$ written in the reverse order. Finally, for a pair of letters $x$ and $y$ in a word $w$, let $w|_{\{x,y\}}$ be the subword induced by the letters $x$ and~$y$.  For example, if $w=42535214421$ then $\pi(w)=42531,\ r(w)=12441253524,$ and $w|_{\{4,5\}}=45544$.

Let $k\geq 0$. A graph $G=(V,E)$ is {\em $k$-$11$-representable} if there exists a word $w$ over the alphabet $V$ such that the word $w|_{\{x,y\}}$  contains in total at most $k$ occurrences of the factors in $\{xx,yy\}$ if and only if $xy$ is an edge in $E$. Such a word $w$ is called $G$'s {\em $k$-$11$-representant}. 
Note that $0$-$11$-representable graphs are precisely word-representable graphs, and that $0$-$11$-representants are precisely word-representants.
A graph $G=(V,E)$ is {\em permutationally $k$-$11$-representable} if it has a $k$-$11$-representant that is a concatenation of permutations of $V$.
 The ``11'' in ``$k$-$11$-representable'' refers to counting occurrences of the {\em consecutive pattern} 11 in the word induced by a pair of letters $\{x,y\}$, which is exactly the total number of occurrences of the factors in $\{xx,yy\}$.  

A {\em uniform} (resp., {\em $t$-uniform}) representant of a graph $G$ is a word, satisfying the required properties, in which each letter occurs the same (resp., $t$) number of times. It is known that each word-representable graph has a uniform representant~\cite{KP08}, 
the class of 2-uniformly representable graphs is exactly the class of circle graphs~\cite{KL15}, while the class of $2$-uniformly $1$-$11$-representable graphs is the 
class of interval graphs~\cite{CKKKP2019}.  Interestingly,  2-uniformly representable graphs appear in the literature under the name of ``{\em alternance graph}'', and other names,  in \cite{B72,B94,EQ71,F78,G80} well before the introduction of word-representable graphs; see \cite{B94} for a discussion and more references on alternance graphs.
The main result in~\cite{CKKKP2019} is the following theorem.

\begin{thm}[\cite{CKKKP2019}]\label{begin-end} Every graph $G$ is permutationally $2$-$11$-representable.
\end{thm}

So, when understanding whether each graph is $k$-$11$-representable for a fixed $k$, the only open case to study is $k=1$.

\subsection{Known tools to study 1-11-representable graphs}\label{known-tools-sec}

\noindent
Clearly, each word-representable graph  is 1-11-representable. Indeed, if $w$ is a word-representant of $G$ then, for instance,  $ww$ or  $r(\pi(w))w$ are its 1-11-representants. There are three types of tools 
for finding 1-11-representable graphs suggested in~\cite{CKKKP2019}:

\begin{itemize}
\item modifying known 1-11-representable graphs;
\item removing edges from word-representable graphs;
\item adding vertices to certain classes of graphs.
\end{itemize}

\begin{figure}
\begin{center}
\begin{tabular}{c c c c c }
\begin{tikzpicture}[scale=0.5]
			
\draw (2,8) node [scale=0.5, circle, draw](node1){$v_1$};
\draw (4,8) node [scale=0.5, circle, draw](node2){$v_{2}$};
\draw (6,8) node [scale=0.5, circle, draw](node3){$v_{3}$};
\draw (2,5) node [scale=0.5, circle, draw](node4){$u_{1}$};
		
\draw (4,5) node [scale=0.5, circle, draw](node5){$u_{2}$};
\draw (6,5) node [scale=0.5, circle, draw](node6){$u_{3}$};
\draw (4,3)node [scale=0.5, circle, draw](node7){$x$};
		
\draw (node1)--(node2) -- (node3);
\draw [bend right=40] (node3) to (node1);
\draw (node1)--(node5);
\draw (node1)--(node6);
		
\draw (node2)--(node4);
\draw (node2)--(node6);
\draw (node3)--(node4);
\draw (node3)--(node5);
\draw (node7)--(node4);
\draw (node7)--(node5);
\draw (node7)--(node6);
			
\end{tikzpicture}
		
&

&
				
\begin{tikzpicture}[scale=0.5]
		
\draw (2,8) node [scale=0.5, circle, draw](node1){$v_1$};
\draw (4,8) node [scale=0.5, circle, draw](node2){$v_{2}$};
\draw (6,8) node [scale=0.5, circle, draw](node3){$v_{3}$};
\draw (8,8) node [scale=0.5, circle, draw](node4){$v_{4}$};
\draw (2,5) node [scale=0.5, circle, draw](node5){$u_{1}$};
		
\draw (4,5) node [scale=0.5, circle, draw](node6){$u_{2}$};
\draw (6,5) node [scale=0.5, circle, draw](node7){$u_{3}$};
\draw (8,5)node [scale=0.5, circle, draw](node8){$u_4$};
\draw (5,2)node [scale=0.5, circle, draw](node9){$x$};
		
\draw (node1)--(node2) -- (node3) -- (node4);
\draw [bend right=30] (node4) to (node1);

\draw (node1)--(node8);
\draw (node4)--(node5);
\draw (node1)--(node6);
\draw (node2)--(node7);
		
\draw (node3)--(node8);
\draw (node2)--(node5);
\draw (node3)--(node6);
\draw (node4)--(node7);
\draw (node9)--(node5);
\draw (node9)--(node6);
\draw (node9)--(node7);
\draw (node9)--(node8);
		
\end{tikzpicture}
		
& 
			
&
		
\begin{tikzpicture}[scale=0.5]
			
\draw (2,8) node [scale=0.5, circle, draw](node1){$v_1$};
\draw (4,8) node [scale=0.5, circle, draw](node2){$v_{2}$};
\draw (6,8) node [scale=0.5, circle, draw](node3){$v_{3}$};
\draw (8,8) node [scale=0.5, circle, draw](node4){$v_{4}$};
			
\draw (10,8) node [scale=0.5, circle, draw](node5){$v_{5}$};
\draw (2,5) node [scale=0.5, circle, draw](node6){$u_{1}$};
\draw (4,5)node [scale=0.5, circle, draw](node7){$u_{2}$};
\draw (6,5)node [scale=0.5, circle, draw](node8){$u_{3}$};
\draw (8,5)node [scale=0.5, circle, draw](node9){$u_{4}$};
\draw (10,5)node [scale=0.5, circle, draw](node10){$u_{5}$};
\draw (6,3)node [scale=0.5, circle, draw](node11){$x$};
		
\draw (node1)--(node2) -- (node3) -- (node4) -- (node5);
\draw [bend right=25] (node5) to (node1);
\draw (node1)--(node7);
\draw (node2)--(node8);
		
\draw (node3)--(node9);
\draw (node4)--(node10);
\draw (node6)--(node2);
\draw (node7)--(node3);
\draw (node8)--(node4);
\draw (node9)--(node5);
\draw (node11)--(node6);
\draw (node11)--(node7);
\draw (node11)--(node8);
\draw (node11)--(node9);
\draw (node11)--(node10);
\draw (node1)--(node10);
\draw (node6)--(node5);
			
\end{tikzpicture}
			
\end{tabular}
\end{center}
\caption{The graphs $\mu(C_3)$, $\mu(C_4)$, and $\mu(C_5)$ }
\label{Mycielski}
	
\end{figure}
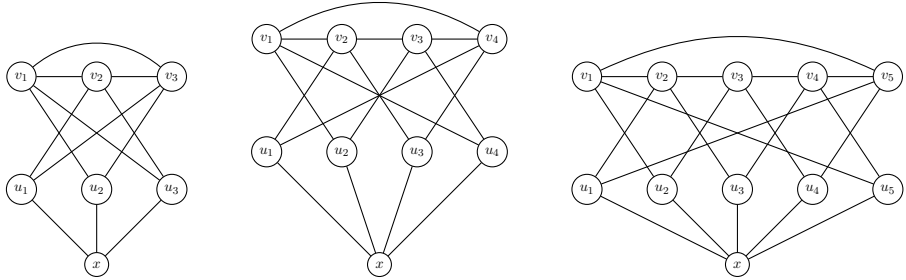

Below we unify all known tools from~\cite{CKKKP2019} into three statements according to their type.

\begin{lem}[\cite{CKKKP2019}]\label{type1-lem}  \ 

(a) Let $G_1$ and $G_2$ be $1$-$11$-representable graphs. Then their disjoint union, glueing them in a vertex or connecting them by an edge results in a $1$-$11$-representable graph.

(b) If $G$ is $1$-$11$-representable then for any edge $xy$ adding a new vertex adjacent to $x$ and $y$ only, gives a $1$-$11$-representable graph.
\end{lem}

\begin{lem}[\cite{CKKKP2019}]\label{type2-lem} Let $G$ be  a word-representable graph, $A\subseteq V$ and $v\in V$. Then

(a) $G\setminus \{xy\in E(G)\ |\ x,y\in A\}$ is a $1$-$11$-representable graph;

(b) $G\setminus \{uv\in E(G)\ |\ u\in N_A(v)\}$ is a $1$-$11$-representable graph.
\end{lem}

\begin{lem}[\cite{CKKKP2019}]\label{type3-lem} Let $G$ be a graph with a vertex $v$.  $G$ is $1$-$11$-representable if at least one of the following conditions holds:

(a) $G\setminus v$ is a comparability graph;

(b) $G\setminus v$ is a circle graph.
\end{lem}

Note that the tool in Lemma~\ref{type3-lem}(b) (that is a partial case of Theorem 2.7 in \cite{CKKKP2019} for $k=2$) looks to be the strongest one. For instance, it allows to establish 1-11-representability of such known non-word-representable graphs as odd wheels. In the next statement we use it to prove a new result on 1-11-representability of $\mu(C_n)$. Note 
that $\mu(C_n)$ is conjectured to be non-word-representable for all odd $n\ge 3$, and it is known that the conjecture is true for $\mu(C_5)$ \cite{KP23}. 

\begin{prop}
The Mycielski graphs $\mu(C_n)$ are $1$-$11$-representable for all $n\ge 3$.
\end{prop}

\begin{proof}
By Lemma~\ref{type3-lem}(b) it is sufficient to show that the graph $\mu(C_n)\setminus x$ is a circle graph, i.e.\ that it is 2-uniformly representable. It is easy to check that the following 2-uniform word represents $\mu(C_n)\setminus x$:
$$v_2u_1u_2v_1v_3u_2u_3v_2v_4\ldots v_iu_{i-1}u_iv_{i-1}v_{i+1}u_iu_{i+1}v_i\ldots v_nu_{n-1}u_nv_{n-1}v_1u_nu_1v_n.$$
Indeed, it is easy to see that the 2-uniform word $v_2v_1v_3v_2\ldots v_nv_{n-1}v_1v_n$ represents the cycle $C_n$. The $u$'s are inserted into this word in such a way that between two copies of $u_i$ one finds only $v_{i-1}$ and $v_{i+1}$ for every $i$ (including the cyclical shifts of the word with the indices $0=n$ and $n+1=1$).  So, $N(u_i) =\{ v_{i-1}, v_{i+1}\}= N_{C_n}(v_i)$, as required.

\end{proof}

\begin{figure}

\begin{center}

\begin{tabular}{ccc}

\begin{tikzpicture}[node distance=1cm,auto,main node/.style={circle,draw,inner sep=2.5pt,minimum size=4pt}]

\node[main node] (9) {{\tiny 9}};
\node[main node] (8) [left of=9] {{\tiny 8}};
\node[main node] (10) [below right of=9] {{\tiny 10}};
\node[main node] (7) [below left of=8] {{\tiny 7}};
\node[main node] (6) [below of=7] {{\tiny 6}};
\node[main node] (5) [below right of=6] {{\tiny 5}};
\node[main node] (11) [below of=10] {{\tiny 11}};
\node[main node] (12) [below left of=11] {{\tiny 12}};
\node[main node] (2) [above left of=8,xshift=-1cm,yshift = 0.2cm] {{\tiny 2}};
\node[main node] (3) [above right of=9,xshift=1cm,yshift = 0.2cm] {{\tiny 3}};
\node[main node] (1) [below left of=5,xshift=-1cm,yshift = -0.2cm] {{\tiny 1}};
\node[main node] (4) [below right of=12,xshift=1cm,yshift = -0.2cm] {{\tiny 4}};

\path
(5) edge (6)
     edge (12)
     edge (9);
\path
(7) edge (6)
     edge (8)
     edge (11);
\path
(9) edge (8)
     edge (10);
\path
(11) edge (10)
     edge (12)
     edge (7);
\path
(8) edge (12);
\path
(6) edge (10);

\path
(1) edge (2)
     edge (4)
     edge (12)
     edge (7);

\path
(3) edge (2)
     edge (4)
     edge (11)
     edge (8);
\path
(2) edge (6)
     edge (9);

\path
(4) edge (5)
     edge (10);

\end{tikzpicture}

&
\ \ \
&

\begin{tikzpicture}[node distance=1cm,auto,main node/.style={circle,draw,inner sep=2.5pt,minimum size=4pt}]

\node[main node] (9) {{\tiny 9}};
\node[main node] (8) [left of=9] {{\tiny 8}};
\node[main node] (10) [below right of=9] {{\tiny 10}};
\node[main node] (7) [below left of=8] {{\tiny 7}};
\node[main node] (6) [below of=7] {{\tiny 6}};
\node[main node] (5) [below right of=6] {{\tiny 5}};
\node[main node] (11) [below of=10] {{\tiny 11}};
\node[main node] (12) [below left of=11] {{\tiny 12}};
\node[main node] (2) [above left of=8,xshift=-1cm,yshift = 0.2cm] {{\tiny 2}};
\node[main node] (3) [above right of=9,xshift=1cm,yshift = 0.2cm] {{\tiny 3}};
\node[main node] (1) [below left of=5,xshift=-1cm,yshift = -0.2cm] {{\tiny 1}};
\node[main node] (4) [below right of=12,xshift=1cm,yshift = -0.2cm] {{\tiny 4}};

\path
(1)  [->,>=stealth', shorten >=1pt] edge (12);
\path
(2)  [->,>=stealth', shorten >=1pt] edge (1);
\path
(3)  [->,>=stealth', shorten >=1pt] edge (1)
       edge (2)
       edge (8)
       edge (11);
\path       
(4)  [->,>=stealth', shorten >=1pt] edge (1)
       edge (2)
       edge (3)
       edge (5)
       edge (10);
\path
(5)  [->,>=stealth', shorten >=1pt] edge (12);
\path
(6)  [->,>=stealth', shorten >=1pt] edge (2)
       edge (5)
       edge (7)
       edge (10);
\path
(7)  [->,>=stealth', shorten >=1pt] edge (1)
       edge (8)
       edge (11);
\path
(8)  [->,>=stealth', shorten >=1pt] edge (12);
\path
(9)  [->,>=stealth', shorten >=1pt] edge (2)
       edge (5)
       edge (8)
       edge (10);
\path
(10)  [->,>=stealth', shorten >=1pt] edge (11);
\path
(11)  [->,>=stealth', shorten >=1pt] edge (12);

\end{tikzpicture}

\end{tabular}
\end{center}

\caption{The Chv\'{a}tal graph (to the left) and a semi-transitive orientation of the Chv\'{a}tal graph extended by the edges 31 and 42 (to the right)}\label{Chvatal-graph}
\end{figure}
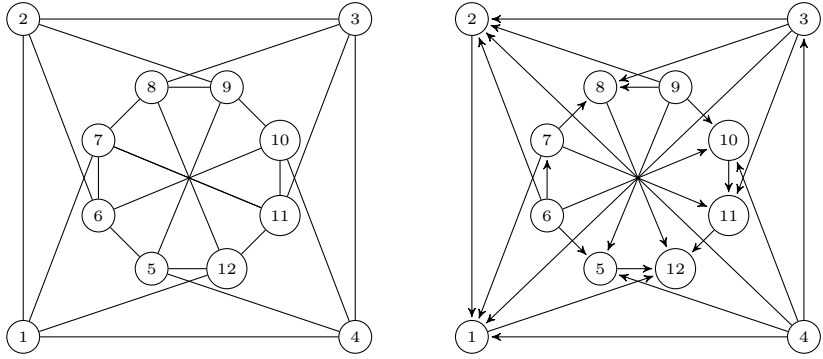

\section{New tools to study 1-11-representation of graphs and their applications}\label{new-tool-sec}

\noindent
Our first tool (Theorem~\ref{new-tool-thm} below) is a far-reaching generalization of Lemma~\ref{type2-lem}. We begin with the following easy observation.

\begin{prop}\label{three-perms-prop} Let $\Pi_1$, $\Pi_2$, $\Pi_3$ be three permutations over $[n]=\{1,\ldots,n\}$. Then the word $w=\Pi_1\Pi_2\Pi_3$ permutationally $1$-$11$-represents the graph with the vertex set $[n]$ in which $x$ and $y$ are not connected by an edge if and only if  in  $\Pi_1$ and $\Pi_3$, $x$ and $y$ are in the same relative order, while in $\Pi_2$ they are in the opposite order.\end{prop}

\begin{proof} We may assume that $x<y$ in $\Pi_1$. Then the word $w|_{\{x,y\}}$ is either one of $xyxyxy,\ xyxyyx,\ xyyxyx$ (then $xy$ is an edge) or $xyyxxy$ (then $x$ and $y$ are not adjacent).  \end{proof}

In the proof of the next theorem, and in other places in the rest of the paper, for convenience, we slightly abuse the notation by denoting a set $A$ and a certain permutation of elements in $A$ by the same letter. This will not cause any confusion. 

\begin{thm}\label{new-tool-thm} Let $V_1,\ldots,V_k$ be pairwise disjoint subsets of $[n]$, the set of vertices of a word-representable graph $G$. We denote by $E(V_i)$ the set of all edges of $G$ having both end-points in $V_i$.  Then, the graph $H=G\backslash(\cup_{1\leq i\leq k} E(V_i))$, obtained by removing all edges belonging to $E(V_i)$ for all $1\leq i\leq k$, is $1$-$11$-representable. \end{thm}

\begin{proof}  Let $w$ be a word representing $G$ and recall that $\pi(w)$ denotes the initial permutation of $w$.  By \cite{KP08}, we can assume that $w$ is uniform.  Also, we let $R:=[n]\backslash(\cup_{1\leq i\leq k} V_i)$ and we define the permutation $\Pi_1:=V_1V_2\ldots V_kR$, where all letters in each subset follow the same order as they have in $\pi(w)$.  Let $\Pi_2:=r(V_1)r(V_2)\ldots r(V_k)R$. We will next prove that the word $W=\Pi_1\Pi_2\pi(w)ww$ 1-11-represents\footnote{In fact, the shorter word $\Pi_1\Pi_2ww$ also represents the graph $H$, but we inserted $\pi(w)$ for the convenience of the reader, making it easier to  follow our arguments.} the graph $H$.

Note that the word $\pi(w)ww$ 1-11-represents $G$ 
and since $w$ is uniform, each edge of $G$ is represented in $w$ by strict alternation of letters (avoiding occurrences of the pattern 11). Clearly, all non-edges in $G$ remain non-edges in $H$.

If $xy$ is an edge in $G$ that belongs to $E(V_i)$ for some $i$, then by Proposition~\ref{three-perms-prop}, $(\Pi_1\Pi_2\pi(w))|_{\{x,y\}}$ contains at least two occurrences of the patterns 11, and hence $x$ and $y$ are not connected by an edge in $H$.

Suppose that $xy$ is an edge in both $G$ and $H$. Hence, $x$ and $y$ cannot belong to any $V_i$. But then in the permutations $\Pi_1$ and $\Pi_2$ the letters $x$ and $y$ are in the same order.  By Proposition~\ref{three-perms-prop}, the word $(\Pi_1\Pi_2\pi(w))|_{\{x,y\}}$ contains at most one occurrence of the pattern 11. As it was shown above, the word  $(\pi(w)ww)|_{\{x,y\}}$ has no such occurences. So, $W|_{\{x,y\}}$ has at most  one occurrence of the pattern 11, which is consistent with  $xy$ being an edge in $H$. 
\end{proof}

A particular case of Theorem~\ref{new-tool-thm}, when each $V_i$ is of size 2, is useful from an applications point of view and hence is stated as a separate result.

\begin{cor}\label{adding-matching} Let the graph $G$ be obtained from a graph $H$ by adding a matching (that is, by adding new edges no pair of which shares a vertex). If $G$ is word-representable then $H$ is $1$-$11$-representable. \end{cor}

\begin{figure}
\begin{center}
\begin{tabular}{c c c c c c }
\begin{tikzpicture}[scale=0.5]
			
\draw (2,5) node [scale=0.7, circle, draw](node1){$1$};
\draw (5,8) node [scale=0.7, circle, draw](node2){$2$};
\draw (9,8) node [scale=0.7, circle, draw](node3){$3$};
\draw (12,5) node [scale=0.7, circle, draw](node4){$4$};
\draw (9,2) node [scale=0.7, circle, draw](node5){$5$};
\draw (5,2) node [scale=0.7, circle, draw](node6){$6$};
\draw (8,5)node [scale=0.7, circle, draw](node7){$7$};
			
\draw (node1)--(node2);
\draw (node2)--(node3);
\draw (node3)--(node4);
\draw (node4)--(node5);
\draw (node5)--(node6);
\draw (node6)--(node7);
\draw (node6)--(node1);
\draw (node2)--(node7);
\draw (node7)--(node4);
\end{tikzpicture}

&

&

&

&
	
\begin{tikzpicture}[node distance=1.2cm,auto,main node/.style={circle,draw,inner sep=2.5pt,minimum size=4pt}]

\node[main node] (1) {{\tiny 1}};
\node[main node] (2) [above of=1] {{\tiny 2}};
\node[main node] (3) [right of=2] {{\tiny 3}};
\node[main node] (4) [below of=3] {{\tiny 4}};
\node[main node] (5) [below left of=2, yshift=2mm] {{\tiny 5}};
\node[main node] (6) [above right of=2, xshift=-2mm] {{\tiny 6}};
\node[main node] (7) [above right of=3] {{\tiny 7}};
\node[main node] (8) [below left of=1] {{\tiny 8}};

\path
(1) edge (2)
     edge (3)
     edge (4)
     edge (5)
     edge (8);

\path
(2) edge (3)
     edge (4)
     edge (5)
     edge (6)
     edge (7)
     edge (8);

\path
(3) edge (4)
     edge (6)
     edge (7);
     
\path
(4) edge (7)
     edge (8);

\end{tikzpicture}

\end{tabular}
\end{center}
\caption{The graph $BW_3$ and a minimal non-word-representable split graph} \label{BW3}
\end{figure}
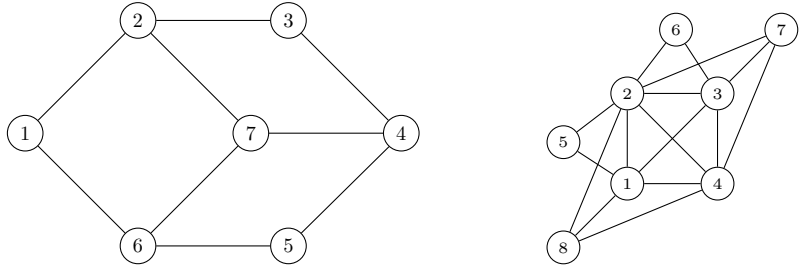

\subsection{The Chv\'{a}tal graph is 1-11-representable}\label{Chvatal-sub}

\noindent
The Chv\'atal graph, given to the left in Figure~\ref{Chvatal-graph}, is the smallest triangle-free 4-chromatic 4-regular graph on 12 vertices \cite{C1970}. This graph is non-word-representable \cite{KP23}.
Firstly,  we show that no known tool from \cite{CKKKP2019} can be applied for proving its 1-11-representability.

\begin{prop}\label{prop-non-applicability}
$1$-$11$-representability of the Chv\'atal graph does not follow from Lemmas~\ref{type1-lem},~\ref{type2-lem}, and~\ref{type3-lem}.
\end{prop}

\begin{proof}
It is evident that Lemma~\ref{type1-lem} cannot be applied. 

Assume that Lemma~\ref{type2-lem} can be applied, i.e.\ that there is a word-representable graph $G,$ its vertex subset $A$ and a vertex $v$ such that $G\setminus E'$
is the Chv\'atal graph where either $E'= \{xy\in E(G)\ |\ x,y\in A\}$ or $E'= \{uv\in E(G)\ |\ v\in N_A(v)\}$.
Consider a semi-transitively oriented copy of $G$ (that exists by Theorem~\ref{semi-trans-thm}) and remove from it the edges in $E'$.
The obtained oriented graph must contain a shortcut $S$ since  the Chv\'atal graph is not word-representable \cite{KP23}. Since  the Chv\'atal graph  is triangle-free, $S$ must contain a directed path $u_1\rightarrow u_2\rightarrow u_3\rightarrow u_4$ with edges $u_1u_3$ and $u_2u_4$ missing. However, none of the variants of $E'$ can
simultaneously contain the edges $u_1u_3$ and $u_2u_4$ and miss the edges $u_1u_2, u_2u_3,$ and $u_3u_4$. Hence, $S$ must be a shortcut in $G$, a contradiction. 

Finally, let us show that Lemma~\ref{type3-lem} cannot be applied. Since the Chv\'atal graph contains two non-intersecting cycles of length 5 induced by the sets $\{1,2,3,7,8\}$ and $\{5,6,10,11,12\}$ (see the left graph in Figure~\ref{Chvatal-graph} for the notations), removing any vertex in the graph cannot produce a comparability graph. Moreover,  it is known \cite{B94} that a circle graph cannot contain a graph $BW_3$ (the left one in Figure~\ref{BW3}) as an induced subgraph. It is straightforward to verify that each of the subsets $V_1=\{ 2,3,5,6,8,9,12\}$,
$V_2=\{ 3,4,6,7,8,10,11\}$, $V_3=\{ 1,4,5,8,9,10,12\}$, and $V_4=\{ 1,2,6,7,10,11,12\}$ induces a copy of $BW_3$. Since $V_1\cap V_2\cap V_3\cap V_4=\emptyset$, the Chv\'atal graph cannot be turned into a circle graph by removing one vertex.
\end{proof}

\begin{rem}
Note that the same arguments as those in Proposition~\ref{prop-non-applicability} for non-applicability of  Lemma~\ref{type2-lem} work not only for the Chv\'atal graph, but for any triangle-free graph. 
\end{rem}

However, the new tool from Theorem~\ref{new-tool-thm} works well for the Chv\'{a}tal graph.

\begin{thm}
The Chv\'{a}tal graph is $1$-$11$-representable.
\end{thm}

\begin{proof}
Add to the Chv\'{a}tal graph the edges 13 and 24 and consider the orientation of the obtained graph $G$ presented to the right in Figure~\ref{Chvatal-graph}. It is easy to verify that this orientation is acyclic. Assume that it has a shortcut. Note that a shortcut must contain a  path of length at least 3. There are exactly seven such paths in $G$, namely, 
$$9\rightarrow 10 \rightarrow 11\rightarrow 12,\ \ \  6\rightarrow 10 \rightarrow 11\rightarrow 12,\ \ \ 6\rightarrow 7 \rightarrow 11\rightarrow 12,$$
$$6\rightarrow 7 \rightarrow 8\rightarrow 12,\ \ \ 4\rightarrow 10 \rightarrow 11\rightarrow 12,\ \ \ 4\rightarrow 3 \rightarrow 11\rightarrow 12,$$
$$4\rightarrow 3 \rightarrow 2\rightarrow 1.$$ First six of them are not shortcuts since the vertex $12$ is not adjacent to  4, 6, or 9. The last one is not a shortcut since the subgraph induced by the vertices $1,2,3,4$ is transitive. So, the orientation of $G$ is semi-transitive and by Corollary~\ref{adding-matching} 
 the Chv\'{a}tal graph is 1-11-representable. 
\end{proof}

\subsection{1-11-representability of split graphs and their generalizations}\label{split-sub}

\noindent
Our second tool is a new technique of finding permutational 1-11-representants for certain graphs. We first present the technique for split graphs and then generalize it to a class of graphs that can be partitioned into an independent set and a comparability graph. However, we believe that the new technique could be applicable  
in proving 1-11-representability of other classes of graphs.

Studying word-representation of split graphs is a hard problem, and it has been the subject of interest in \cite{CKS21,I2021,KLMW17,KP24}. It is remarkable that each split graph is 1-11-representable as is shown in the following theorem.

\begin{thm}\label{split-thm} Any split graph is permutationally $1$-$11$-representable. \end{thm}

\begin{proof} Let $A=\{a_1,\ldots,a_k\}$ be a clique and $B=\{b_1,\ldots,b_{\ell}\}$ be an independent set in a split graph $S$, so that $A\cup B$ is the set of all vertices in~$S$. For a vertex $a_i\in A$ let $N_i=N_B(a_i)$ (resp., $O_i=B\backslash N_i$) be the set of neighbours (resp., non-neighbours) of $a_i$ in $B$. We put
$$w_0:=a_1a_2\ldots a_kb_1b_2\ldots b_{\ell}\ a_1a_2\ldots a_k  b_{\ell} b_{\ell-1}\ldots b_1\ a_1a_2\ldots a_kb_1b_2\ldots b_{\ell}$$
and define the permutations
\begin{eqnarray}
\Pi_k&:=& a_1a_2\ldots a_{k-1}O_ka_kN_k; \nonumber\\ 
\Pi_j&:=& a_ka_{k-1}\ldots a_{j+1}a_1a_2\ldots a_{j-1}O_ja_jN_j, \mbox{ for }0<j<k; \nonumber \\ 
\Pi_0&:=& a_ka_{k-1}\ldots a_1b_1b_2\ldots b_{\ell}. \nonumber
\end{eqnarray}

Then the word $w=w_0\Pi_k\Pi_{k-1}\ldots \Pi_0$ permutationally 1-11-represents the graph $S$.  

Indeed, the factor $w_0$ of $w$ ensures independence of the set $B$. Moreover, for each pair $a_i,a_j\in A$ where $i<j$ in 
 $w|_{a_i,a_j}$ 
we have a subsequence $a_ia_ja_ia_j\ldots a_ia_j$ to the left of the permutation $\Pi_j$ (including $\Pi_j$ itself), and a subsequence   $a_ja_ia_ja_i\ldots a_ja_i$ to the right of $\Pi_j$. So,
there is exactly one occurence of the pattern 11 in  $w|_{a_i,a_j}$ 
 ensuring that $a_i$ and $a_j$ are connected. Next, suppose that $a_i\in A$ and $b\in B$. If $a_i$ is adjacent to $b$, then $w|_{\{a_i,b\}}=a_iba_ib\ldots a_ib$, which has no pattern 11. Finally,  if $a_i$ is not adjacent to $b$ then $(w\backslash \Pi_i)|_{\{a_i,b\}}=a_iba_ib\ldots a_ib$ but $\Pi_i|_{\{a_i,b\}}=ba_i$, so  $w|_{\{a_i,b\}}$ has two occurrences of the pattern 11 that is consistent with $a_i$ being not adjacent to $b$. 

Thus, $w$ 1-11-represents $G$. Since $w_0$ is a concatenation of three permutations, $w$ is also a concatenation of permutations.
\end{proof}

To illustrate the construction in the proof of Theorem~\ref{split-thm}, we give a permutational 1-11-representation of the split graph given in Figure~\ref{BW3} to the right  that is observed in \cite{KLMW17} to be minimal non-word-representable (removing any of its vertices results in a word-representable graph). We have $A=\{1,2,3,4\}$, $B=\{5,6,7,8\}$, $k=\ell=4$, $N_1=\{5,8\}$, $O_1=\{6,7\}$, $N_2=\{5,6,7,8\}$, $O_2=\emptyset$, $N_3=\{6,7\}$, $O_3=\{5,8\}$, $N_4=\{7,8\}$ and $O_4=\{5,6\}$.  Separating permutations by space for more convenient visual representation, we have:
\begin{eqnarray}
w_0&=&12345678\ 12348765\ 12345678 \nonumber \\
\Pi_4 & = & 123O_44N_4 = 12356478\nonumber \\
\Pi_3 & = & 412O_33N_3= 41258367\nonumber \\
\Pi_2 & = & 431O_22N_2 = 43125678 \nonumber \\
\Pi_1 & = & 432O_11N_1=43267158\nonumber \\
\Pi_0 & = & 43215678\nonumber 
\end{eqnarray}
and so a permutational 1-11-representation of the graph to the right in Figure~\ref{BW3} is 
\begin{small}
$$12345678\ 12348765\ 12345678\ 12356478\ 41258367\ 43125678\ 43267158\ 43215678.$$
\end{small}

The following theorem is a far-reaching generalization of Theorem~\ref{split-thm}. However, we do keep Theorem~\ref{split-thm} as a separate result as we need the construction in its proof in what follows.

\begin{thm}~\label{ind+comp} Suppose that the vertices of a graph $G$ can be  partitioned into a comparability graph formed by vertices in $A=\{a_1,\ldots,a_k\}$ and an independent set formed by vertices in $B=\{b_1,\ldots,b_{\ell}\}$. Then $G$ is permutationally $1$-$11$-representable.  \end{thm}

\begin{proof} Denote by $G'$ the split graph obtained from $G$ by substitution of $A$ by a clique $A'$. By Theorem~\ref{split-thm} $G'$ can be permutationally 1-11-represented by the word 
$w=w_0\Pi_k\Pi_{k-1}\ldots \Pi_1\Pi_0$. Moreover, for each $a_i, a_j\in A'$  the subword $w|_{a_i,a_j}$ contains exactly one occurrence of the pattern 11. 

 By Theorem~\ref{comp-thm}, the subgraph $G[A]$ is permutationally representable. So, let $Q_1Q_2\ldots Q_t$ be its representation by permutations $Q_i$ over the set $A$. Let $\Pi'_i=Q_ib_1b_2\ldots b_{\ell}$ for all $i\in\{1,2,\ldots,t\}$ and rename, if necessary, the vertices in $A$ so that $Q_1=a_ka_{k-1}\ldots a_1$  (i.e.\ so that  $\Pi'_1=\Pi_0$ in the word $w$). We put
$$W=w_0\Pi_k\Pi_{k-1}\ldots \Pi_1\Pi'_1\Pi'_2\ldots \Pi'_t$$  and show that it permutationally 1-11-represents $G$.

Indeed, the factor $w_0\Pi_k\Pi_{k-1}\ldots \Pi_1\Pi'_1$ of $W$ defines the split graph with a clique formed by the vertices in $A$ and an independent set $B$. Also, any edge $a_ib_j$ of the split graph remains an edge in $G$ since the order of these vertices is $a_ib_j$ in all permutations of $W$.

Let $i<j$ and consider vertices $a_i, a_j\in A$. By construction, in the word  $w=w_0\Pi_k\Pi_{k-1}\ldots \Pi_1\Pi'_1$ each edge $a_ia_j$ of the clique $A'$ is defined by the subsequence $$a_ia_ja_ia_j\ldots  a_ia_ja_ja_ia_ja_i\ldots a_ja_i$$ containing exactly one occurrence of the pattern 11. If $a_ia_j$ is an edge of the comparability graph $G[A]$, then in all permutations $\Pi'_s$ vertices $a_i$ and $a_j$ are in the same order $a_ja_i$, and so $\Pi'_1\Pi'_2\ldots \Pi'_t|_{\{a_i,a_j\}}$ avoids the pattern 11 and hence $a_ia_j$ remains an edge in $G$. Finally, if  $a_ia_j$ is not an edge of the comparability graph $G[A]$, then in $\Pi'_1\Pi'_2\ldots \Pi'_t|_{\{a_i,a_j\}}$  we  have at least one occurrence of the pattern 11, and hence $w|_{\{a_i,a_j\}}$ has at least two occurrences of the pattern 11, so in $G$ $a_i$ and $a_j$ are not connected by an edge.
\end{proof}

\subsection{1-11-representability of all graphs on at most 7 vertices}\label{7-verti-sub}

\noindent
1-11-representation of all graphs on at most 7 vertices is established in \cite{CKKKP2019}. However, the arguments in \cite{CKKKP2019} are based on the incorrect list of 25 non-word-representable graphs published in several places in the literature, in particular, in \cite{KL15}. The problem with the list was spotted in \cite{KitaevSun}, and the two incorrect graphs, Graphs 12 and 17, were replaced in \cite{KitaevSun} by the correct graphs given in Figure~\ref{two-graphs}. Hence, technically, 1-11-representation of all graphs on at most 7 vertices, but Graph 12 and Graph 17, is known, and next we complete the classification by confirming 1-11-representability of the graphs in  Figure~\ref{two-graphs}.

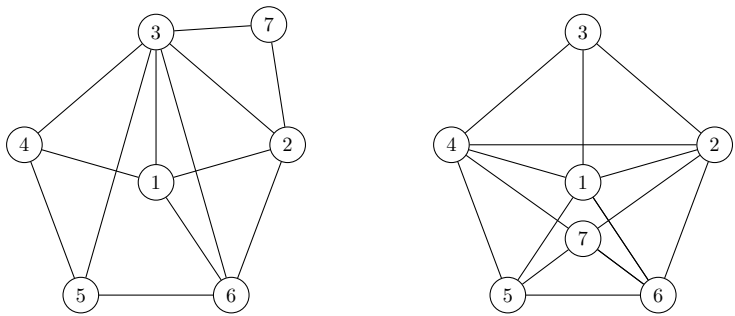
\begin{figure}
\begin{center}
\begin{tabular}{c c c c c c }
\begin{tikzpicture}[scale=0.5]
			
\draw (7,5) node [scale=0.7, circle, draw](node1){$1$};
\draw (10.5,6) node [scale=0.7, circle, draw](node2){$2$};
\draw (7,9) node [scale=0.7, circle, draw](node3){$3$};
\draw (3.5, 6) node [scale=0.7, circle, draw](node4){$4$};
\draw (5,2) node [scale=0.7, circle, draw](node5){$5$};
\draw (9,2) node [scale=0.7, circle, draw](node6){$6$};
\draw (10,9.2)node [scale=0.7, circle, draw](node7){$7$};
			
\draw (node1)--(node2);
\draw (node1)--(node3);
\draw (node1)--(node4);
\draw (node2)--(node3);

\draw (node3)--(node4);
\draw (node3)--(node5);
\draw (node4)--(node5);
\draw (node5)--(node6);
\draw (node6)--(node1);
\draw (node6)--(node2);
\draw (node6)--(node3);
\draw (node2)--(node7);
\draw (node7)--(node3);

\end{tikzpicture}

&

&

&

&
	
\begin{tikzpicture}[scale=0.5]
			
\draw (7,5) node [scale=0.7, circle, draw](node1){$1$};
\draw (10.5,6) node [scale=0.7, circle, draw](node2){$2$};
\draw (7,9) node [scale=0.7, circle, draw](node3){$3$};
\draw (3.5, 6) node [scale=0.7, circle, draw](node4){$4$};
\draw (5,2) node [scale=0.7, circle, draw](node5){$5$};
\draw (9,2) node [scale=0.7, circle, draw](node6){$6$};
\draw (7,3.5)node [scale=0.7, circle, draw](node7){$7$};

\draw (node1)--(node5);
\draw (node1)--(node2);
\draw (node1)--(node3);
\draw (node1)--(node4);
\draw (node6)--(node1);
\draw (node2)--(node3);
\draw (node2)--(node4);
\draw (node3)--(node4);
\draw (node4)--(node5);
\draw (node5)--(node6);
\draw (node6)--(node1);
\draw (node6)--(node2);
\draw (node6)--(node7);
\draw (node6)--(node7);
\draw (node7)--(node2);
\draw (node5)--(node7);
\draw (node7)--(node4);

\end{tikzpicture}

\end{tabular}
\end{center}
\caption{Non-word-representable Graph 12 (to the left) and Graph 17 (to the right)}\label{two-graphs}
\end{figure}

\begin{prop}
Graphs $12$ and $17$ are permutationally $1$-$11$-representable.
\end{prop}

\begin{proof}
Note that removing the independent set $\{1,5,7\}$ from Graph~12 results in a triangle with a pending edge, that is a comparability graph. Similarly, removing  the independent set $\{1,7\}$ from Graph~17 results in a 5-cycle with a chord, that is also a comparability graph.  So, by Theorem~\ref{ind+comp} both graphs are permutationally 1-11-representable. 
\end{proof}

Note that there exist shorter non-permutational 1-11-representants for these graphs found using software:
$$w_{12}=4573275465142631256\ \ \ \ \ w_{17}= 23474625731436251645.$$

\section{Concluding remarks}\label{final-sec}

\noindent
In this paper we introduce new tools to study 1-11-representable graphs, which allows to confirm 1-11-representability of Chv\'{a}tal graph,  Mycielski graph,  split graphs and graphs whose vertices can be partitioned into a comparability graph and an independent set. Finally, we confirm a claim in \cite{CKKKP2019} that all graphs on at most 7 vertices are 1-11-representable.

It is still an open problem whether each graph is 1-11-representable. Moreover, it is still unknown whether each graph is permutationally 1-11-representable, and towards constructing potential counterexamples, one should look for a graph for which none of the known existing tools is applicable.   Note that even if all graphs are (permutationally) 1-11-representable, the constructions of 1-11-representations presented in this paper can still be useful for finding explicit representations of graphs, with an aim towards potential applications. \\

\noindent
{\bf Acknowledgement.} The first author acknowledges the PICME scholarship from CNPQ, which funded him throughout the preparation of this paper. The second author is grateful to the SUSTech International Center for Mathematics for its hospitality during his visit to the Center in April 2024. The work of the third author was partially supported by the state contract of the Sobolev Institute of Mathematics (project FWNF-2022-0019). The authors are grateful to the unknown referee for many useful comments. \\

\noindent
{\bf Declaration.}  The authors declare no conflict of interest.


\begin{thebibliography}{99}
\bibitem{B72} A. Bouchet. Caract\'{e}risation des symboles crois\'{e}s de genre nul, {C. R. Acad. Sci. Paris} {\bf 274} (1972), 724--727.
\bibitem{B94} A. Bouchet. Circle graph obstructions, {\em J. Combin. Theory, Ser. B} {\bf 60} (1) (1994), 107--144.
\bibitem{CKS21} H. Z.Q. Chen, S. Kitaev, A. Saito. Representing split graphs by words, {\em Discussiones Mathematicae Graph Theory} {\bf 42} (2022) Article ID: 4314, p.1263.
\bibitem{CKKKP2019}  G.-S. Cheon, J. Kim, M. Kim, S. Kitaev and A. Pyatkin. On $k$-11-representable graphs, {\em J. Combin.} {\bf 10} (2019) 3, 491--513.
\bibitem{C1970} V. Chv\'{a}tal. The smallest triangle-free 4-chromatic 4-regular graph, {\em J. Combin. Theory} {\bf 9} (1970) 1, 93--94.
\bibitem{EQ71} S. Even and A. Itai. Queues, stacks and graphs, ``{\em Theory of Machines and Computations}'', pp. 71--86, Academic Press, New York, 1971.
\bibitem{F78} J. C. Fournier. Une caract\'{e}risation des graphes de cordes. {\em C. R. Acad. Sci. Paris} {\bf 286A} (1978), 811--813.
\bibitem{G80} M. C. Golumbic. ``{\em Algorithm Graph Theory}'', Academic Press, New York, 1980.
\bibitem{I2021} K. Iamthong. Word-representability of split graphs generated by morphisms. {\em Discrete Appl. Math.} {\bf 314} (2022) 284--303.
\bibitem{JKPR} M. Jones, S. Kitaev, A. Pyatkin and J. Remmel. Representing graphs via pattern avoiding words, {\em Electronic J. Combin.} {\bf 22(2)} (2015), \#P2.53, 20 pp..
\bibitem{Halldorsson} M. M. Halld\'{o}rsson, S. Kitaev and A. Pyatkin. Semi-transitive orientations and word-representable graphs. {\em Discr. Appl. Math.} {\bf 201} (2016) 164--171.
\bibitem{HS1981} P.L. Hammer and Simeone, B. The Splittance of a Graph, {\em Combinatorica} {\bf 1} (1981) 275--284.
\bibitem{KLMW17} S. Kitaev, Y. Long, J. Ma, H. Wu. Word-representability of split graphs, {\em J. of Combin.} {\bf 12} (2021) 4, 725--746.
\bibitem{KL15} S. Kitaev and V. Lozin. Words and Graphs, {\em Springer}, 2015.
\bibitem{KP08} S. Kitaev and A. Pyatkin. On representable graphs. {\em J. Autom., Lang. and Combin.} {\bf 13} (2008) 1, 45--54.
\bibitem{KP23} S. Kitaev and A. Pyatkin. On semi-transitive orientability of triangle-free graphs, {\em Discussiones Mathematicae Graph Theory} {\bf 43} (2023), ID: 4621, page 533.
\bibitem{KP24} S. Kitaev and A. Pyatkin. On semi-transitive orientability of split graphs, {\em Information Processing Letters} {\bf 184} (2024), 106435.
\bibitem{KitaevSun} S. Kitaev and H. Sun. Human-verifiable proofs in the theory of word-representable graph, {\em RAIRO -- Theoretical Informatics and Appl.}, to appear.
\bibitem{LB1962} C.G. Lekkerkerker and J.C. Boland. Representation of a finite graph by a set of intervals on the real line, {\em Fundamenta Mathematicae} {\bf 51} (1962), 45--64.
\bibitem{Myc} J. Mycielski, Sur le coloriage des graphes, {\em Colloq. Math.} {\bf 3 (2)} (1955), 161--162.
\bibitem{S1994} J.P. Spinrad. Recognition of circle graphs, {\em Journal of Algorithms} {\bf 16(2)} (1994), 264--282.
\bibitem{Spinrad} J.P. Spinrad. Efficient Graph Representations, {\em Fields Institute Monographs}, Volume: 19; 2003; 342 pp. (ISBN 978-1-4704-3146-4).
\end{thebibliography}
\end{document}